\documentclass{llncs}
\usepackage{amsmath}
\usepackage{amssymb}
\usepackage{authblk}
\usepackage{graphicx}
\newtheorem{prop}{Proposition}

\newtheorem{cor}[prop]{Corollary}

\newtheorem{thm}[prop]{Theorem}

\usepackage{bbding}

\usepackage{color}
\usepackage{xcolor,colortbl}
\usepackage{appendix}
\renewcommand{\phi}{\varphi}
\newcommand{\D}{\mathcal{D}}
\newcommand{\N}{\mathbb{N}}

\newcommand{\cs}{2^\omega}
\newcommand{\fs}{2^{<\omega}}
\newcommand{\MLR}{\ensuremath{\mathsf{MLR}}}
\newcommand{\PA}{\ensuremath{\mathsf{PA}}}
\newcommand{\K}{\ensuremath{\mathrm{K}} }
\newcommand{\C}{\ensuremath{\mathrm{C}} }
\newcommand{\HH}{\ensuremath{\mathrm{H}} }

\newcommand{\Omeg}{\mathrm{\Omega}}

\newcommand{\uh}{\upharpoonright}

\newcommand{\aaa}{\mathbf{a}}
\newcommand{\pp}[2]{\; {| \hspace{-1.7mm} \sim^{#2}_{#1}\, }}

\begin{document}

\title{The axiomatic power of Kolmogorov complexity}
\author{Laurent Bienvenu \inst{1}
\and Andrei Romashchenko \inst{2}
\and Alexander Shen 
\inst{2}
\and Antoine Taveneaux \inst{1}
\and Stijn Vermeeren \inst{3}
}

\date{}
\institute{LIAFA, CNRS \& Universit\'e Paris 7
\and LIRMM, CNRS \& Universit\'e Montpellier 2 
\and University of Leeds
}

\maketitle
\let\eps=\varepsilon

\begin{abstract}
The famous G\"odel incompleteness theorem states that for every consistent sufficiently rich formal theory~$T$ there exist true statements that are unprovable in~$T$. Such statements would be natural candidates for being added as axioms, but how can we obtain them? One classical (and well studied) approach is to add to some theory $T$ an axiom that claims the consistency of~$T$. In this paper we discuss another approach motivated by Chaitin's version of G\"odel's theorem where axioms claiming the randomness (or incompressibility) of some strings are probabilistically added, and show that it is not really useful, in the sense that this does not help us to prove new interesting theorems. This result (cf.~\cite{euler2006}) answers a question recently asked by Lipton~\cite{Lipton-blog}. The situation changes if we take into account the size of the proofs: randomly chosen axioms may help making proofs much shorter (unless NP=PSPACE). This result partially answers the question asked in~\cite{euler2006}. 

We then study the axiomatic power of the statements of type ``the Kolmogorov complexity of $x$ exceeds $n$'' (where $x$ is some string, and $n$ is some integer) in general. They are $\mathrm{\Pi}_1$ (universally quantified) statements of Peano arithmetic. We show (Theorem~\ref{thm:all-complexities}) that by adding all true statements of this type, we obtain a theory that proves all true $\mathrm{\Pi}_1$-statements, and also provide a more detailed classification. In particular, as Theorem~\ref{thm:finite-power} shows, to derive all true $\mathrm{\Pi}_1$-statements it is enough to add one statement of this type for each~$n$ (or even for infinitely many~$n$) if strings are chosen in a special way. On the other hand, one may add statements of this type for most~$x$ of length $n$ (for every $n$) and still obtain a weak theory (Theorem~\ref{a_lot_is_not_enough}). We also study other logical questions related to ``random axioms'' (hierarchy with respect to $n$, Theorem~\ref{thm:generalized-chaitin} in Section~\ref{subsec:axpower}, independence in Section~\ref{subsec:independence}, etc.).

Finally, we consider a theory that claims Martin-L\"of randomness of a given \emph{infinite} binary sequence. This claim can be formalized in different ways. We show that different formalizations are closely related but not equivalent, and study their properties. 
\end{abstract}

\section{Introduction}

We assume that the reader is familiar with the notion of Kolmogorov complexity and Martin-L\"of randomness (See~\cite{LiV2008,Shen-uppsala,DowneyH2010} for background information about Kolmogorov complexity and related topics), but since for our purposes this notion needs to be expressed in formal arithmetic, we recall the basic definitions. The Kolmogorov complexity $\C(x)$ of a binary string $x$ is defined as the minimal length of a program (without input) that outputs $x$ and terminates. This definition depends on a programming language, and one should choose one that makes complexity minimal up to $O(1)$ additive term. Technically, there exist different versions of Kolmogorov complexity. Prefix complexity $\K(x)$ assumes that programs are self-delimiting. We consider plain complexity $\C(x)$ where no such assumptions are made; any partial function $D$ can be used as an ``interpreter'' of a programming language, so $D(p)$ is considered as an output of program $p$, and $\C_D(x)$ is defined as the minimal length of $p$ such that $D(p)=x$. Then some optimal $D$ is fixed (such that $\C_D$ is minimal up to $O(1)$ additive term), and $\C_D(x)$ is called (plain Kolmogorov) complexity of $x$ and denoted $\C(x)$. Most strings of length $n$ have complexity close to $n$. More precisely, the fraction of $n$-bit strings that have complexity less than $n-c$,  is at most $2^{-c}$. In particular, there exist strings of arbitrary high complexity. 

However, as G.~Chaitin pointed out in~\cite{Chaitin1971}, the situation changes if we look for strings of \emph{provably} high complexity. More precisely, we are looking for strings~$x$ and numbers~$n$ such that the statement ``$\C(x)>n$'' (properly formalized in arithmetic; note that this is a $\mathrm{\Pi}_1$ statement) is provable in formal (Peano) arithmetic~\PA. Chaitin noted that there is a constant~$c$ such that \emph{no} statement ``$\C(x)>n$'' is provable in $\PA$ for $n>c$. Chaitin's argument is a version of Berry's paradox: Assume that for every integer $k$ we can find some string $x$ such that ``$\C(x)>k$'' is provable; let $x_k$ be the first string with this property in the order of enumeration of all proofs; this definition provides a program of size $O(\log k)$ that generates $x_k$, which is impossible for large $k$ since ``$\C(x_k)>k$'' is provable in \PA\ and therefore true (in the standard model).\footnote{Another proof of the same result shows that Kolmogorov complexity is actually not very essential here. By a standard fixed-point argument one can construct a program $p$ (without input) such that for every program $q$ (without input) the assumption ``$q$ is equivalent to $p$'' (i.e., $q$ produces the same output as $p$ if $p$ terminates, and $q$ does not terminate if $p$ does not terminate) is consistent with \PA. If $p$ has length $k$, for every $x$ we may assume without contradiction that $p$ produces $x$, so one cannot prove that $\C(x)$ exceeds $k$.}  

This leads to a natural idea. Toss a coin $n$ times to obtain a string $x$ of length~$n$, and consider the statement ``$\C(x)\ge n-1000$''. This statement is true unless we are extremely unlucky. The probability of being unlucky is less than $2^{-1000}$. In  natural sciences we are accustomed to identify this with impossibility. So we can add this statement and be (almost) sure that it is true; if $n$ is large enough, we get a true non-provable statement and could use it as a new axiom. We can even repeat this procedure several times: if the number of iterations $m$ is not astronomically large, $2^{-1000}m$ is still astronomically small.

Now the question: \emph{Can we obtain a richer theory in this way and get some interesting consequences, still being practically sure that they are true}? The answers are given in Section~\ref{sec:pp}: 
\begin{itemize}
\item yes, this is a safe way of enriching \PA\ (Theorem~\ref{correctness});
\item yes, we can get a stronger theory this way (Chaitin's theorem), but
\item no, we cannot prove anything interesting this way (Theorem~\ref{conservation}).
\end{itemize}

So the answer to our question is negative; however, as we show in section~\ref{proof-size} (Theorem~\ref{complexity}), these ``random axioms'' do give some advantages: while they cannot help us to prove new interesting statements, they can significantly shorten some proofs (unless PSPACE=NP).

In Section~\ref{section_finite_string_sequences} we switch to a more general question: what is the axiomatic power of statements ``\,$\C(x)>n$'' for different $x$ and $n$ and how are they related to each other? They are $\mathrm{\Pi}_1$-statements; we show that adding all true statements of the form ``\,$\C(x)>n$'' as axioms, we can prove all true $\mathrm{\Pi}_1$-statements (Theorem~\ref{thm:all-complexities}).\footnote{By ``statements" we mean statements in the language of \PA, and by ``true statements" we mean statements that are true in the standard model of \PA.}

We show that the axiomatic power of these statements increases as $n$ increases, and relate this increase to a classification of $\mathrm{\Pi}_1$-statements by their ``complexity'' (cf.~\cite{CaludeC2009,CaludeC2010}). We show that for some $c$ and for all $n$ one (true) statement $\C(x)>n$ for some string $x$ of length $n$ is enough to prove all true statements $\C(y)>n-c$ (Theorem~\ref{thm:finite-power}), and that all true statements $\C(x)>n$ for given $n$ are not enough to prove any statement $\C(y)>n+c$ (Theorem~\ref{thm:generalized-chaitin}). In other words, the bigger the lower bound for complexity is, the more powerful the axioms we get.

This result is rather fragile. First, the choice of a string $x$ is important: for other strings of length $n$ this property is not true (even if we add many axioms at the same time, see Theorem~\ref{a_lot_is_not_enough}). Also the precision of the complexity information is crucial (Theorem~\ref{thm:half-complexity-not-enough}).

Then we show that a random choice of several random axioms (about complexities) with high probability leads to independent statements (any combination of these statements with or without negations is consistent with \PA, Theorem~\ref{thm:independence}).

Many of the results mentioned are closely related to corresponding statements in computation theory. Theorem~\ref{thm:strange-disjunction} shows, however, that such a correspondence does not necessarily works.

The results mentioned above deal with unconditional complexity. Some of them trivially generalize to conditional complexity, but sometimes the situation changes. In Theorem~\ref{thm:cc} we show that there exists some $c$ such that every true $\mathrm{\Pi}_1$-statement can be derived in \PA\ from true statements of the form $\C(x|y)>c$. (Here $\C(x|y)$ is the length of the shortest program that maps $y$ to $x$.) 


Finally, in Section~\ref{section_martin_lof_random} we switch from finite strings to infinite sequences and consider theories that say (in some exact way) that a given infinite binary sequence $X$ is Martin-L\"of random.

\section{Probabilistic proofs in Peano arithmetic}\label{sec:pp}
\subsection{Random axioms: soundness}\label{sec:prob-proof-strategy}

Let us describe more precisely how we generate and use random axioms. Assume that some initial ``capital'' $\eps$ is fixed. Intuitively, $\eps$ measures the maximal probability that we agree to consider as ``negligible''. 

\medskip

\noindent
\textbf{The basic version:} Let $c$ be an integer such that $2^{-c}<\eps$ and $n$ an integer. We choose at random (uniformly) a string $x$ of length $n$, and add the statement ``\,$\C(x)\ge n-c$'' to $\PA$ (so it can be used together with usual axioms of \PA).\footnote{As usual, we should agree on the representation of Kolmogorov complexity function~$\C$ in \PA. We assume that this representation is chosen in some natural way, so all the standard properties of Kolmogorov complexity are provable in \PA. For example, one can prove in \PA\ that the programming language used in the definition of $\C$ is universal. The correct choice is especially important when we speak about proof lengths (Section~\ref{proof-size}). }

\medskip
\noindent
\textbf{A slightly extended version:} We fix several numbers $n_1,\ldots,n_k$ and $c_1,\ldots,c_k$ such that
$2^{-c_1}+\ldots+2^{-c_k}<\eps$. Then we choose at random strings $x_1,\ldots,x_k$ of length $n_1,\ldots,n_k$,
and add all the statements ``\,$\C(x_i)\ge n-c_i$'' for $i=1,\ldots,k$. 

\medskip
\noindent
\textbf{Final version:} In fact, we can allow an even more flexible procedure of adding random axioms that does not mention Kolmogorov complexity explicitly. Assume that we have already proved for some finite set $A$ of natural numbers, for some rational $\delta>0$ and for some property $R(x)$ (an arithmetical formula with one free variable $x$) that \emph{the proportion of members of $A$ such that $\lnot R(n)$ is at most $\delta$}. We denote the latter statement by $(\forall_\delta x\in A)R(x)$ in the sequel. This statement can be written in \PA\ in a natural way. We assume here that $A$ is represented by some formula $A(x)$; for each $n\in A$ the formula $A(\bar n)$ is provable in $\PA$; for each $n\notin A$ the formula $\lnot A(\bar n)$ is provable in $\PA$; finally, we assume that the cardinality of $A$ is known in $\PA$ (the corresponding formula is provable).

Then we are allowed to pick an integer $n$ at random inside $A$ (uniformly), and add the formula $R(\bar n)$ as a new axiom. This step can be repeated several times. We have to pay $\delta$ for each operation until the initial capital $\eps$ is exhausted; different operations may have different values of $\delta$. Note that the axiom added at some step can be used to prove the cardinality bound at the next steps.\footnote{Actually this is not important: we can replace previously added axioms by conditions. Assume for example that we first proved $(\forall_\delta x\in A) R(x)$, then for some randomly chosen $n$ we used $A(n)$ to prove $(\forall_\tau y \in B) S(y)$, and finally we added $S(m)$ for some random $m\in B$. Instead, we can prove without additional axioms the statement $(\forall_\tau y\in B)\, (A(n)\to S(y))$, and add $A(n)\to S(m)$, which gives the same result, since $A(n)$ is added earlier in the proof.}

Our previous examples are now special cases: $A$ is the set of all $n$-bit strings (formally speaking, the set of corresponding integers, but we identify them with strings in a natural way), the formula $R(x)$ says that $\C(x)\ge n-c$, and $\delta=2^{-c}$. 

\smallskip
In this setting we consider \emph{proof strategies} instead of proofs. Normally, a proof is a sequence of formulas where each next formula is either an axiom or is obtained from previous ones by some inference rule. Now we have also random steps where we go from the formula $(\forall_\tau x\in A) R(x)$ to some $R(n)$ for randomly chosen $n\in A$. Formally, a proof strategy is a finite tree whose nodes are labeled by pairs $(T,\delta)$, where $T$ is a set of formuli (obtained so far) and $\delta$ is a rational between $0$ and $1$ (the capital at that stage of the process). Nodes can be of two types:

\begin{itemize}
\item Deterministic nodes only have one child. If $(T,\delta)$ is the label of a deterministic node, the label $(T',\delta')$ of its child has $\delta'=\delta$ (no capital is spent), and $T'=T \cup \{\psi\}$ where $\psi$ is some statement which can be obtained from~$T$ using some inference rule.  
\item Probabilistic nodes have several children. If $(T,\delta)$ is a probabilistic node, the set $T$ contains a formula $(\forall_\tau n \in A)\, R(n)$ where $\tau\le\delta$, the node has as many children as elements of $A$, and the labels of the children are $(T \cup \{R(n)\},\delta-\tau)$ where $n$ ranges over~$A$.   
\end{itemize}

The root of the tree has label $(\PA,\varepsilon)$, where $\PA$ is the set of axioms for Peano arithmetic, and $\varepsilon$ is the initial capital.

We will often identify a proof strategy $\pi$ with the corresponding probabilistic process that starts from the root of the tree, and at every probabilistic node chooses uniformly a child of the current node, until a leaf is reached. The logical theory~$T$ built at the end of the process (when a leaf is reached) is, in this context, \emph{a random variable}. We call it \emph{the theory built by~$\pi$}. 

\begin{figure}[h]
\begin{center}
     \includegraphics{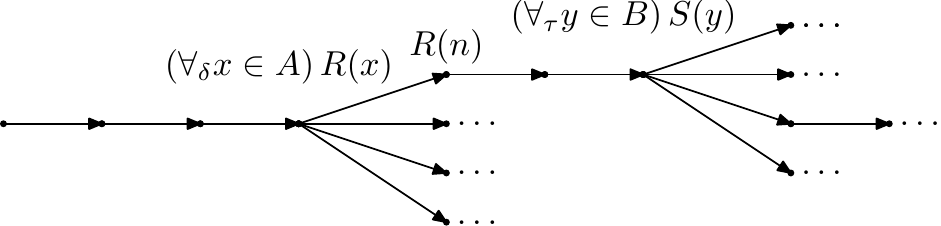}
\end{center}
\caption{A proof strategy represented as a tree}\label{proof-tree}
\end{figure}

Given a proof strategy $\pi$ and a formula $\phi$, we consider the probability that~$\phi$ is provable by $\pi$, i.e., the total probability of all leaves where $\phi$ appears. 

One can consider also a compressed version of the proof where we omit all non-branching steps. In this version each vertex is labeled with some theory; the root has label \PA; at a non-leaf vertex some formula $(\forall_\delta x\in A)R(x)$ is provable in the corresponding theory, and at its sons the formulas $R(n)$ for different $n\in A$ are added to that theory. The probability of $\varphi$ being provable by this strategy is the probability of leaves where $\varphi$ is provable.

We can also inductively define the relation $T  \pp{p}{\varepsilon} \phi$ which means that there exists a randomized proof starting from $T$ with capital $\varepsilon\ge 0$ that proves $\phi$ with probability at least $p$. The inductive steps are:

\[
\frac{T \vdash \phi}{T  \pp{p}{\varepsilon} \phi}\quad 
\text{for every $\varepsilon\ge 0$ and $p\in [0,1]$;}\quad
\frac{}{T  \pp{0}{\varepsilon} \phi}\quad
\text{for every $\varepsilon\ge 0$;}
\]

\[
\frac{%
T \vdash (\forall_\delta x\in A)\,R(x)
\qquad 
\left\{T,R(\bar n) \pp{p_n}{\varepsilon-\delta} \phi \right\}_{n \in A}}
{T  \pp{p}{\varepsilon} \phi}
\quad
\text{if $\delta\le\varepsilon$, $0\le p\le \sum_{n\in A}p_i /\#A$}
\]

It is easy to see that this inductive definition is equivalent to the original one: \emph{$T  \pp{p}{\varepsilon} \varphi$ if and only if there exists a proof strategy with initial capital $\varepsilon$ that proves $\varphi$ starting from $T$ with probability $p$ or more}. 
Indeed, if $T\vdash \varphi$, then every capital $\varepsilon\ge0$ is enough to prove $\varphi$, we do not need randomized steps; everything is provable with probability at least $0$; if $T \vdash (\forall_\delta x\in A)\,R(x)$, then we can start the proof from $T$ using a randomized step, and the probability to prove $\varphi$ is the average of the probabilities $p_n$ to prove it starting from the enlarged theories using the remaining capital $\varepsilon-\delta$. 

On the other hand, having a proof strategy, we may use the backward induction (from leaves to the root) to establish the $\pp{p}{\varepsilon} \varphi$-relation for all the nodes of the tree (for current capital $\varepsilon$ and the probability $p$ to prove $\varphi$ starting from the current vertex, or any smaller number).

Having all these equivalent definitions, we nevertheless consider the first definition of a proof strategy as the main one. This is important when we speak about the length of the proof. Note also that we do not assume here that a proof strategy is effective  in any sense. 

The following theorem says that this procedure can indeed be trusted:

\begin{thm}[soundness]\label{correctness}
Let~$\pi$ be a proof strategy with initial capital~$\delta$. The probability that the theory~$T$ built by~$\pi$ contains some false statement is at most~$\delta$. 
\end{thm}

This theorem has the following immediate corollary. 

\begin{cor}\label{soundness-corollary}
Let $\phi$ be some arithmetical statement.
If the probability to prove~$\phi$ for a proof strategy $\pi$ with initial capital $\eps$ is greater than $\eps$, then $\phi$ is true. In other words, $\mathrm{PA}\pp{p}{\varepsilon}\varphi$ for $p>\varepsilon$ implies that $\varphi$ is true.
\end{cor}

\begin{proof}[of Theorem~\ref{correctness}]
This theorem is intuitively obvious because a false statement appears only if one of the ``bad events'' has happened (a false axiom was selected at some step), and the sum of probabilities of all bad events that happen along some branch is bounded by $\varepsilon$. However, this argument cannot be understood literally since different branches have different bad events (after the first branching point). 

So let us be more formal and say that a node of label $(T,\delta)$ is \emph{good} if all statements in $T$ are true and \emph{bad} otherwise. We prove the following property by backward induction: $(\clubsuit)$ if a node has a label~$(T,\delta)$, either (a) it is bad, or (b) it is good and the probability that, starting from that node, one will reach a bad node is smaller or equal to~$\delta$. Backward induction means that we prove this property for the leaves (base case) and then prove that if it holds for all children of a node, it also holds for that node. It follows immediately that the property holds for all nodes of the tree, hence it holds at the root, which is what we want (since the root is good).

Here the base case is immediate: all leaves have the property~$(\clubsuit)$; there is nothing to prove. Now suppose that we have a node~$u$ of label $(T,\delta)$ such that all of its children have the property~$(\clubsuit)$. If $u$ is a deterministic node, there is again nothing to prove, as it is easy to see from the definition that a deterministic node has the~$(\clubsuit)$ property if and only if its child does. If~$u$ is probabilistic, let $(\forall_\tau x \in A) R(x)$ be the formula in~$T$ associated to the probabilistic choice at node~$u$. If $u$ is bad, we are done (the property~$(\clubsuit)$ automatically holds at $u$), so let us assume it is good. This means in particular that the formula $(\forall_\tau x \in A) R(x)$ is true, which means that~$u$ has a most a fraction $\tau$ of bad children. By the induction hypothesis, starting from any good child of $u$, the probability to reach a bad node is at most $\delta-\tau$. Thus the probability to reach a bad node starting from~$u$ is at most
\[
\tau + (1-\tau)(\delta-\tau) = \delta - \tau(\delta-\tau) \leq \delta
\]
(the last inequality uses the fact that $\delta	\geq \tau$).\qed
\end{proof} 

\subsection{Random axioms are not useful}

As Chaitin's theorem shows, there are proof strategies that with high probability lead to \emph{some} statements that are true but non-provable (in \PA). However, the situation changes if we want to get some \emph{fixed} statement, as the following theorem  shows:

\begin{thm}[conservation]\label{conservation}
Let $\phi$ be some arithmetical statement.
If the probability to prove $\phi$ for a proof strategy $\pi$ with initial capital $\eps$ is greater than $\eps$, then $\phi$ is provable \textup(in \PA\ without any additional axioms\textup).
\end{thm}

Formally, Theorem~\ref{conservation} is a stronger version of Corollary~\ref{soundness-corollary}, but the message here is quite different: Corollary~\ref{soundness-corollary} is the good news (probabilistic proof strategies are safe) whereas Theorem~\ref{conservation} is the bad news (probabilistic proof strategies are useless). 

\begin{proof}
Let $\phi$ be a fixed statement that is not provable in~$\PA$. We say that a node $u$ is \emph{strong} if it has label $(T,\delta)$ with $T \vdash \phi$, and is \emph{weak} otherwise. We again use a proof by induction, and we prove that each node~$u$ has the property $(\diamondsuit)$: either (a) $u$ is strong or (b) $u$ is not strong and the probability that starting from~$u$ one hits a strong node~$v$ is bounded by $\delta$. 

Again, the fact that leaves all have the property ($\diamondsuit$) is immediate. Now suppose that we have a node~$u$ of label $(T,\delta)$ such that all of its children have the property~$(\diamondsuit)$. If $u$ is a deterministic node, there is again nothing to prove, as it is easy to see from the definition that a deterministic node has the~$(\diamondsuit)$ property if and only if its child does. If~$u$ is probabilistic, let $(\forall_\tau x \in A) R(x)$ be the formula in~$T$ associated to the probabilistic choice at node~$u$. If $u$ is strong, we are done (the property~$(\diamondsuit)$ automatically holds at $u$), so let us assume it is not. Let $p$ be the probability that one hits a strong node starting from~$u$. Let us show that the fraction of strong nodes among the children of $u$ is at most $\tau$. Let $x_1,\ldots,x_t$ be the elements of $A$ corresponding to the strong children of~$u$.  For each $i=1,\ldots, t$ we have $T \cup R(x_i)\vdash \phi$, or equivalently $T \vdash R(x_i)\rightarrow \phi$ by definition of strong vertices. Therefore,
\[
T \vdash  [R(x_1)\lor R(x_2)\lor\ldots\lor R(x_t)] \rightarrow \phi.
\]
If $t > \tau \#A$, this fact, together with the assumption that $(\forall_\tau x \in A) R(x)$, entails $T \vdash \phi$, so we get a contradiction. Thus the fraction of strong children of~$u$ is at most~$\tau$. Moreover, for any weak child~$v$ the induction hypothesis tells us that the probability to hit a strong node starting from~$v$ is at most $\delta-\tau$. Thus the total probability to hit a strong node starting from~$u$ is at most $$\tau + (1-\tau)(\delta-\tau)  \leq \delta.$$
\qed
\end{proof}

\subsection{Polynomial size proofs}\label{proof-size}

The situation changes drastically if we are interested in the length of proofs. The argument used in Theorem~\ref{conservation} gives an exponentially long ``conventional'' proof compared with the original ``probabilistic'' proof, since we need to combine the proofs for all terms in the disjunction. (Here the length of a probabilistic proof strategy is measured as the length of the longest branch; note that the total size of the proof strategy tree may be exponentially larger.) Can  we find another construction that transforms probabilistic proof strategies into standard proofs with only polynomial increase in length? Probably not; some reason for this is provided by the following Theorem~\ref{complexity}.

\begin{thm}\label{complexity}
If every probabilistic proof strategy $\pi$ can be transformed into a deterministic proof whose length is polynomial in the length of~$\pi$, then the complexity classes $\mathrm{PSPACE}$ and $\mathrm{NP}$ coincide.
\end{thm}

\begin{proof}
It is enough to consider a standard $\mathrm{PSPACE}$-complete language, the language TQBF of true quantified Boolean formulas. A standard interactive proof for this language (see, e.g., \cite{Sipser1996}; we assume that reader is familiar with that proof) uses an Arthur--Merlin protocol where  the verifier (Arthur) is able to perform polynomial-size computations and generate random bits that are visible to the prover (Merlin) but cannot be corrupted by him. The correctness of this scheme is based on simple properties of finite fields. This kind of proof can be easily transformed into a successful probabilistic proof strategy in our sense. To explain this transformation, let us recall some details. 

Assume that a quantified Boolean formula $\phi$ starting with a universal quantifier is given. This formula is transformed into a statement that says that for some polynomial $P(x)$ two values $P(0)$ and $P(1)$ are equal to $1$. This polynomial is implicitly defined by a sequence of operations. To convince Arthur that it is indeed the case (i.e., that $P(0)=P(1)=1$), Merlin shows $P$ to Arthur (listing explicitly its coefficients; there are polynomially many of them). In other terms, Merlin notes that the formula $$[\forall x (P(x)=\bar{P}(x))]\rightarrow [P(0)=1 \land P(1)=1]$$ and therefore $$[\forall x(P(x)=\bar{P}(x))]\rightarrow \phi$$ are true (and provable in \PA). Here $P(x)$ is the polynomial $P$ defined as the result of the sequence of operations, while $\bar{P}(x)$ is the explicitly given expression for~$P$. Then Merlin notes that the implication $$[P(r)=\bar{P}(r)]\rightarrow \forall x(P(x)=\bar{P}(x))$$ is true for most elements $r$ of the finite field and this fact is provable in \PA\ (using basic results about finite fields), so such an implication can be added as a new axiom, and it remains to convinve Arthur that $P(r)=\bar P(r)$. Continuing in this way, we get a probabilistic proof strategy that mimics the interactive proof protocol for TQBF.\footnote{This argument assumes that our version of \PA\ allows us to convert the argument above into a polynomial-size randomized proof. We do not go into these technical details here.}

This construction gives for each TQBF a probabilistic proof strategy (in the sense of Section~\ref{sec:prob-proof-strategy}) of polynomial length that uses some small initial capital $\varepsilon$. Assume that any probabilistic proof strategy can be transformed into a conventional proof in \PA\ of polynomial size. Then this conventional proof is a $\mathrm{NP}$-witness for TQBF, so $\mathrm{PSPACE}=\mathrm{NP}$. \qed
\end{proof}

\section{Non-randomly chosen axioms}\label{section_finite_string_sequences}

\subsection{Full information about complexities and its axiomatic power}

In this section we study in general the axiomatic power of the axioms ``$\C(x)>n$". Let us start with a simple question: assume that we add to \PA\ all true statements of this form as axioms. What theory do we get? Note that these statements are $\mathrm{\Pi}_1$-formulas, so their negations are existential formulas and are provable when true. So with these axioms we have the full information about Kolmogorov complexity of every binary string. The axiomatic power of this information has a simple description:

\begin{thm}\label{thm:all-complexities}
If one adds to $\PA$ all true statements of the form ``\,$\C(x)>n$", the resulting theory proves all true $\mathrm{\Pi}_1$-statements.
\end{thm}

\begin{proof}
The proof is an adaptation of the proof that $\mathbf{0}'$ is Turing-reducible to the function $\C$ (see Proposition~2.1.28 in~\cite{Nies2009}). 

Consider the upper bound $\C^t(x)$ for complexity that appears if we restrict the computation time for decoding by $t$. The value of $\C^t(x)$ is computable given $x$ and $t$. As $t$ increases, $\C^t(x)$ decreases (or remains the same), and the limit value is $\C(x)$.
   
For some number $N$, we consider the minimal value of $t$ such that $\C^t(x)$ reaches $\C(x)$  for all strings $x$ of length at most $N$. Let us denote this value by $B(N)$. Let us show that every terminating program (without input) of size $N-O(\log N)$ terminates in at most $B(N)$ steps. Indeed, a terminating program $p$ can be considered as a description of a natural number, namely, the number of steps needed for its termination. Assume that this number is greater than $B(N)$. Then, knowing $p$ and $N$, we can compute the table of true complexities of all strings of length $N$ (since for these strings $\C$ coincides with $\C^t$, where $t$ is the number of steps needed for $p$ to terminate). Therefore we can  effectively find a string of length $N$ that has complexity at least $N$ (such a string always exists).  This leads to a contradiction if the total information in $p$ and $N$ is less than $N-O(1)$; and this is guaranteed if the size of $p$ is less than $N-O(\log N)$. (Note that we need $O(\log N)$ additional bits to specify $N$ in addition to $p$.)

This shows that the information about the complexities of strings of length $N$ is enough to solve the halting problem for all programs of size $N-O(\log N)$. So the halting problem is decidable with $\C$ as oracle (the result mentioned above).

It remains to note that this argument can be formalized in \PA. Using axioms that guarantee the complexity of all strings up to length $N$, we can prove the value of $B(N)$. Also we can prove that each program $p$ of size at most $N-O(\log N)$ terminates in $B(N)$ steps or does not terminate at all. Therefore, we can prove that $p$ does not terminate if it is the case. It remains to note that every $\mathrm{\Pi}_1$-statement is provably equivalent to non-termination of some program (namely, the program looking for a counterexample for this statement).\qed
\end{proof}

\begin{rem}

  \textbf{1}. To be precise, in the last theorem we need to specify how the function $\C$ is represented by an arithmetic formula. It is enough to assume that $\C$ is defined as the minimal length $\C_D(x)$ of the program that produces $x$ with respect to an interpreter $D$, where the interpreter $D$ is provably optimal, i.e., for all $D'$, there exists a constant~$c$ such that 
   $$\PA\vdash \forall x\, \bigl[\C_D(x)\le \C_{D'}(x)+c\bigr].$$ \label{provable-optimality}
We always assume that $\C$ is represented in this way.

\textbf{2}. A sceptic could (rightfully) claim that the exact value of Kolmogorov complexity can encode some additional irrelevant information (in particular, about $\mathrm{\Pi}_1$-statements' truth values). For example, it is not obvious \emph{a priori} that the theory obtained by adding the full information about the values of $\C$ does not depend on the choice of optimal programming language fixed in the definition of Kolmogorov complexity. Also one may ask whether a similar result is true for prefix complexity. 

To address all these questions, we may consider weaker axioms. Assume, for example, that for every $x$ the complexity of $x$ is guaranteed up to a factor~$2$, i.e., some axiom $\C(x)>c_x$ is added where $c_x$ is at least half of the true complexity of~$x$. Is it enough to add these axioms for all $x$ to prove all true $\mathrm{\Pi}_1$-statements? The answer is positive, and a similar argument can be used. Let $B(N)$ to be the minimal $t$ such that $\C^{t}(x)\le 2c_x$ for all strings $x$ of length~$N$. Then, knowing $N$ and any number~$t$ greater than $B(N)$, we can compute $\C^t(x)$ for all strings~$x$ of length $N$, and get a lower bound $c_x \geq \C^t(x)/2$ for $\C(x)$. Taking $x$ such that $\C^t(x)\ge N$ for this $t$ (such an~$x$ exists since $\C^t$ is an upper bound for $\C$), we find a string of length $N$ that has complexity at least $N/2$ (such a string always exists). 

Therefore, every program of length less than $N/2-O(\log N)$ terminates before $B(N)$ steps, otherwise we could use the termination time to get a contradiction.  Knowing (from the added axioms) the value of $\C(x)$ up to factor $2$, we can prove for some $t$ that this is the case, and therefore prove non-termination for non-terminating programs (as before). 
\end{rem}

\subsection{Complexity of $\mathrm{\Pi}_1$-statements}

Let us introduce the notion of complexity of a (closed) $\mathrm{\Pi}_1$-statement that can be considered as a formal version of the ideas of~\cite{CaludeC2009,CaludeC2010}.

Let $U(x)$ be a $\mathrm{\Pi}_1$-statement in the language of \PA\ with one free variable $x$ (for simplicity we identify strings and natural numbers, so we consider $x$ as a string variable).
We say that $U$ is \emph{universal} if for every closed $\mathrm{\Pi}_1$-statement $T$ there exists some string $t$ such that $$\PA\vdash [T\Leftrightarrow U(t)].\eqno(*)$$
For a universal $U$ we can define the $U$-complexity of a closed $\mathrm{\Pi}_1$-statement $T$ as the minimal length of $t$ that satisfies $(*)$. As usual, the following statement is true:

\begin{thm}\label{thm:calude}
   There exists an optimal $\mathrm{\Pi}_1$-statement $U(x)$ such that the corresponding complexity function is minimal up to $O(1)$ additive term.
\end{thm}
\begin{proof}
Let $V(p,x)$ be a $\mathrm{\Pi}_1$-statement with two parameters such that for every $\mathrm{\Pi}_1$-statement $W(x)$ with one parameter there exists some string $p$ such that 
$$\PA\vdash \forall x\, [W(x)\Leftrightarrow V(p,x)].$$ 
Then we let $U(\bar px)=V(p,x)$ where $\bar p$ is some self-delimiting encoding of $p$, e.g., $p$ with doubled bits and $01$ added at the end.\qed
\end{proof}

We fix some optimal universal statement $U(x)$ and measure the complexity of $\mathrm{\Pi}_1$-statements with respect to $U(x)$; the complexity function is well defined up to $O(1)$ additive term.

Evidently, the complexity of all false $\mathrm{\Pi}_1$-statements is the same constant; the same is true for all provable $\mathrm{\Pi}_1$-statements. Since one can construct a computable sequence of non-equivalent (in \PA) closed $\mathrm{\Pi}_1$-statements, it is easy to see that the number of non-equivalent statements of complexity at most $n$ is $\Theta(2^n)$. 

Informally speaking, $\mathrm{\Pi}_1$-statement of complexity at most $n$ are statements about non-termination of programs of complexity at most $n$, and their provable equivalents. (The exact formulation allows $O(1)$-change in $n$.)

\subsection{The axiomatic power of the complexity table up to length $n$}
\label{subsec:axpower}

Now we can describe the axiomatic power of the complexity table up to some $n$: it is roughly equivalent to all true $\mathrm{\Pi}_1$-statements of complexity at most $n$. This informal description requires several clarifications. 

First, the complexity (both for strings and for $\mathrm{\Pi}_1$-statements) is defined only up to $O(1)$ additive term. To take this into account, we say that two sequences $T_1,T_2,\ldots$ and $S_1,S_2.\ldots$ of theories are $O(1)$- equivalent if $T_n\subset S_{n+c}$ and $S_n\subset T_{n+c}$ for some $c$ and all~$n$. (The inclusion $U\subset V$ means that every theorem of~$U$ is a theorem of $V$; we could also write $V\vdash U$.)

Second, we need to specify what we mean by a ``complexity table up to $n$''. There are several possibilities. We may consider axioms that give full information about complexities of all strings of \emph{length} at most $n$. Or we may consider axioms that specify the list of all strings of \emph{complexity} at most $n$ and complexities of all these strings. All these variants work; one can even consider one specific string of length $n$, the lexicographically first string $r_n$ of complexity at least $n$, and consider a theory with only one axiom $\C(r_n)\ge n$. In the following theorem we consider these ``minimal'' and ``maximal'' versions and show that they are essentially equivalent.

\begin{thm}\label{thm:finite-power}
The following sequences of theories are $O(1)$-equivalent:
\begin{description}
\item[$A_n$]\textup: $\C(r_n)\ge n$\textup;
\item[$B_n$]\textup: the list of all strings that have complexity at most $n$, and full information about their complexities \textup(for each string of complexity at most $n$ we add an axiom specifying its complexity, and also add an axiom that says that all other strings have complexity greater than $n$\textup{);}
\item[$C_n$]\textup: all true $\mathrm{\Pi}_1$-statements of complexity at most $n$.
\end{description}
\end{thm}

In particular, adding axioms $\C(r_n)\ge n$ for infinitely many $n$, we get $\PA$ plus all true $\mathrm{\Pi}_1$-statements.

\begin{proof}
It is easy to see that $B_n$ implies $A_n$. 

To prove that $C_{n+c}$ implies $A_n$, consider a $\mathrm{\Pi}_1$-statement $\text{Rand}(r)$ which says that $\C(r)\ge |r|$ (where $|r|$ is the length of $r$). Then $A_{n}$ is $\text{Rand}(r_n)$ and therefore $A_{n}$ consists of a $\mathrm{\Pi}_1$-statement of complexity at most $n+O(1)$. 

It remains to show that $A_n$ implies $B_{n-c}$ and $C_{n-c}$. This is similar to the proof of Theorem~\ref{thm:all-complexities}. 

\begin{lemma}
   \label{finite-power_lemma}
   Let $A(p)$ be some algorithm with input $p$. There exists some $c$ such that for every $n$ and for every input string $p$ of length at most $n-c$ such that $A(p)$ does not terminate, the theory $A_n$ proves non-termination of $A(p)$.
\end{lemma}

This lemma shows that $A_n$ implies $B_{n-c}$ since $A_n$ allows to prove non-termination of all non-terminating programs of size at most $n-c$, so the true values of complexities can be guaranteed if they do not exceed $n-c$. Similarly, $A_n$ implies $C_{n-c}$, since $\mathrm{\Pi}_1$-statement $U(x)$ claims that the program with input~$x$ that searches for the counterexample to this statement, never terminates.

It remains to prove the lemma.

We know that $\C(r_n)\ge n$ and for all preceding strings $y$ of length $n$ we have $\C(y)<n$. Consider the minimal $t$ such that $\C^t(y)<n$ for all these $y$. We denote this value by $B(n)$. Let us prove that for a suitable $c$ (that does not depend on $n$) and for every string $p$ of length at most $n-c$ the computation $A(p)$ either does not terminate or terminates in at most $B(n)$ steps. 

Every string $p$ determines the number of steps needed for the termination of $A(p)$. Knowing $p$ and $n$, we find this number $t(p)$ and then take the first~$y$ of length $n$ such that $\C^{t(p)}(y)\ge n$. If $t(p)$ exceeds $B(n)$, then we get $r_n$. On the other hand, for every $p$ such that $A(p)$ terminates and for every $n$ we get some string of length $n$ whose complexity does not exceed $\C(p,n)+O(1)$, where $\C(p,n)$ stands for the Kolmogorov complexity of the pair $(p,n)$. Note that $\C(p,n)$ is bounded by $|p|+O(\log (n-p))+O(1)$ (we add to $p$ a prefix that is a self-delimiting description of $n-p$). If $c$ is large enough, for every string $p$ of length $n-c$ or less we get a contradiction (assuming that $t(p)$ steps are not enough for the termination of $A(p)$). So all computations $A(p)$ for $|p|\le n-c$ terminate in $B(n)$ steps.

This reasoning can be formalized in \PA. Having $\C(r_n)\ge n$ as an axiom, we can prove that $r_n$ is the first string of complexity at least $n$: all the preceding strings have a short description, and we can wait long enough to confirm that it is indeed the case. Then we can prove the value of $B(n)$ and prove that every computation $A(p)$ for short $p$ either terminates in $B(n)$ steps or does not terminate at all.\qed
\end{proof}

\begin{rem}
To be closer to the initial framework, we may fix some constant $c$ and for every $n$ consider the lexicographically first string $y$ of length $n$ such that $\C(y)\ge n-c$. Then we add to \PA\ the statement $\C(y)\ge n-c$ for this $y$ and get a theory $A_n^c$. For this theory the statement of Theorem~\ref{thm:finite-power} is also true (and can be proved in the same way); of course, the constant in $O(1)$-equivalence depends on $c$.\\
\end{rem}

This theorem leaves the following question open. 

\begin{question}
Characterize precisely the true $\mathrm{\Pi}_1$ statements of complexity~$n$ which prove all true $\mathrm{\Pi}_1$ statements of complexity~$n-O(1)$.
\end{question}

\medskip
It is natural to ask whether the power of theories of Theorem~\ref{thm:finite-power} strictly increases as $n$ increases. It is indeed the case, as the following ``generalized Chaitin's theorem'' shows:

\begin{thm}\label{thm:generalized-chaitin}
   There exists some $c$ such that no statement of the form $\C(x)>n+c$ can be proved in $A_n$ for any $n$.
\end{thm}

(Theorem~\ref{thm:finite-power} allows us to replace $A_n$ in the statement by $B_n$ or $C_n$.) 

\begin{proof}
Consider the following program: given the string $r_n$ and some $d$, it starts to look for $\PA$-consequences of $A_n$ saying that $\C(y)>n+d$ for some $y$ and $d$. (Note that $n$ can be reconstructed as the length of $r_n$). When (and if) such a~$y$ is found, it is the output of the program. 

If the program terminates, the complexity of the output is at most $n+O(\log d)$; on the other hand, it is at least $n+d$, so for all such cases we have $n+O(\log d) \ge n+d$, and therefore  $d\le O(1)$. \qed
\end{proof}

This theorem has an interesting consequence that can be formulated without any reference to complexities. Note that for every number $n$ one can write an arithmetic formula with one parameter $x$ that in provably equivalent to $x=n$ and has length $O(\log n)$. (The standard formula with successor function has length $O(n)$, but we can use binary representation.) Using this fact, it is easy to show that complexity of a $\mathrm{\Pi}_1$-formula can be defined up to $O(1)$-factor as the minimal length of provably equivalent formula. We know that one axiom of $A_n$ is enough to prove all axioms of $C_n$. In terms of lengths we get the following statement:

\begin{thm}
    For every $n$ there exists a true $\mathrm{\Pi}_1$-formula of size $O(n)$ that implies in \PA\ every true $\mathrm{\Pi}_1$-formula of size at most $n$.
\end{thm}

\subsection{Not all random strings are equally useful}

Theorem~\ref{thm:finite-power} shows that it is enough to claim the incompressibility of one properly chosen string $r_n$ to derive all true $\mathrm{\Pi}_1$-statements of complexity $n-O(1)$. However, this is a very special property of this incompressible string, as we see in this section.

Recall that there is $\Theta(2^n)$ incompressible strings of length $n$. Indeed, there are at most $2^n-1$ programs of length less than $n$, but some of them are needed to produce longer strings of complexity less than $n$. There is $\Theta(2^n)$ such strings; for example, one can consider strings $xx$ where $x$ is a string of length $n-O(1)$. So we have $\Theta(2^n)$ incompressible strings of length $n$.

The following result shows that we can add many of them and still have a theory that is weaker that theories of Theorem~\ref{thm:finite-power}. We formulate this result for the more general case of $c$-incompressible strings.  (A string $x$ is called \emph{$c$-incompressible} if $\C(x)\ge|x|-c$.)

\begin{thm}\label{a_lot_is_not_enough}
Fix a constant $c$. Let $m(n)$ be a computable provable lower bound for the number of $c$-incompressible strings of length~$n$. For example, we can let $m(n)=2^n-2^{n-c}$, or $m(n)=\varepsilon2^n$ for $c=0$.

 Let $\phi$ be a formula  not provable in $\PA$. Then it is possible to choose $m(n)$ strings of length $n$ \textup(for each $n$\textup) in such a way that with \PA\ with axioms $\C(x)\ge n-c$ for all these $x$ does not prove $\phi$.
\end{thm}

(The strings added are necessarily $c$-incompressible, otherwith the theory is inconsistent and therefore proves~$\phi$.)

\begin{proof}
We choose $m(n)$ strings of length $n$ sequentially for $n=1,2,\ldots$ in such a way that $\phi$ remains unprovable after each step. Assume that we did this for all lengths smaller than $n$ and $\phi$ is unprovable in the resulting theory. Let $t=m(n)$. We want to add $t$ axioms of the form $\C(x)\ge n-c$ so that $\phi$ remains unprovable. Imagine that this is impossible. Then for every $t$ strings $x_1,\ldots,x_t$ one can prove (in the current extension of \PA) that
$$ (\C(x_1)\ge n-c) \land (\C(x_2)\ge n-c) \land\ldots\land (\C(x_t)\ge n-c) \Rightarrow \phi.$$
Note that this is the case for every $x_1,\ldots,x_t$, even if one of the statements $\C(x_i)\ge n-c$ is false, since then the left hand side is provably false. Therefore, we can also prove
$$ \bigvee_{x_1,\ldots,x_t} (\C(x_1)\ge n-c) \land (\C(x_2)\ge n-c) \land\ldots\land (\C(x_t)\ge n-c) \Rightarrow \phi,$$
and the left hand side is provable in \PA\ due to the lower bound $m(n)$. So we conclude that $\phi$  was already provable before the induction step, contrary to the induction assumption.\qed 
\end{proof} 

\begin{rem}
In this proof it is important that $m(n)$ is not just a computable lower bound for the number of $c$-incompressible strings, but a \emph{provable} lower bound. Without this condition ($m$ is only assumed to be a computable lower bound), one can prove a weaker result: it is possible to add $m(n)$ many axioms of type ``$\C(x)>n$" (for all $n$) in such a way that the resulting theory does not prove \emph{all} true $\mathrm{\Pi}_1$-statements. The reason is that the set $\mathcal{C}$ consisting of sequences of statements of type ``$\C(x)>n$" which are consistent with \PA\ and contain at least $m(n)$ elements for all~$n$ is (modulo proper encoding) a $\mathrm{\Pi}^0_1$ subset of $\cs$ (the set of infinite binary sequences). Hence $\mathcal{C}$ must contain an element $S$ (sequence of statements) which does not compute the halting set $\mathbf{0}'$ (this follows from the low basis theorem of Jockusch and Soare~\cite{JockuschS1972}). However, any theory proving all true $\mathrm{\Pi}_1$-statements can Turing-compute $\C$ and therefore can compute~$\mathbf{0}'$. So adding the sequence~$S$  does not allow us to prove all true $\mathrm{\Pi}_1$-statements. This type of argument combining computability theory and logic will be an important tool in Section~\ref{section_martin_lof_random}. 
\end{rem}

\subsection{Usefulness of random axioms is fragile}

Theorem~\ref{thm:finite-power} says that for carefully chosen $r_n$ the incompressibility axioms are rather strong (imply all true $\mathrm{\Pi}_1$-statements)  while for many other incompressible strings this is not the case (Theorem~\ref{a_lot_is_not_enough}).

In this subsection, we show that the usefulness of the well-chosen incompressibility axioms is not solely due to the strings $r_n$ themselves, but also to the accuracy of the axiom. Namely, we prove the following:

\begin{thm}\label{thm:half-complexity-not-enough}
Let $(r_n)$ be a sequence of strings such that $|r_n|=n$ and $\C(r_n) \geq n-O(1)$. There exists a constant~$c$ such that the axioms ``\,$\C(r_n) \geq n-c\log n$'' \textup(for all~$n$\textup) do not prove all true $\mathrm{\Pi}_1$-statements. 
\end{thm}

The first step of the proof of this theorem is the following lemma.\footnote{The idea of proving this lemma came to us by reading an early draft of Higuchi et al.'s paper~\cite{HiguchiHSY} where it was stated as an open problem; by the time we wrote up our proof and informed them of the solution, Higuchi et al.\ had independently solved it. }

\begin{lemma}\label{lem:higuchi}
Let $(r_n)$ be a sequence of strings such that $\C(r_n) \geq n$. Suppose $Y \in \cs$ computes the sequence $(r_n)$, and $Y$ is uniformly Martin-L\"of random with respect to some oracle $Z\in \cs$. Then $\C^Z(r_n)\ge n-c\log n$ for some $c$ and for all $n$. 
\end{lemma}

\begin{proof}
To prove the lemma, it is enough to proof the following inequality:
    $$
\C(r)\le \C^Z(r) + \C^Y(r) + \mathbf{d}^Z(Y) + O(\log \C(r)),    
    $$
where $\mathbf{d}^Z(Y)$ stands for the expectation-bounded randomness deficiency of $Y$ with oracle $Z$.  (This deficiency was introduced by Levin and G\'acs, see~\cite{BienvenuGHRS2011} for details; $\mathbf{d}^Z(Y)$ is finite if and only if $Y$ is Martin-L\"of random with oracle $Z$. In this subsection we assume that the reader is familiar with the definition and properties of randomness deficiency.)

Indeed, for $r=r_n$ the value of $\C(r)$ is at least $n$; the value of $\C^Y(r)$ is $O(\log n)$, since $r_n$ is computable with oracle $Y$ from $n$; the deficiency is finite; finally, the last term is $O(\log n)$.  So we get $\C^Z(r_n)\ge n-O(\log n)$.

It remains to prove the inequality above. We may use prefix complexity $\K$ instead of plain complexity $\C$, since our inequality has logarithmic precision anyway (all complexities are bounded by $\C(r)$ and we have $O(\log\C(r))$ term). So we need to prove that 
 $$
\K(r)\le \K^Z(r) + \K^Y(r) + \mathbf{d}^Z(Y) + O(\log \K(r)),    
    $$
For given $r$ we consider all infinite sequences $\tilde{Y}$ that (being used as oracle) decrease the (prefix) complexity of $r$ from $\K(r)$ to $\K^Y(r)$. The set $W$ of all such sequences is effectively open (since only finite information about an oracle can be used). It contains $Y$ (by construction) and has small measure: we will show that its measure is $O(2^{-s})$ where $s=\K(r)-\K^Y(r)$ is the decrease in complexity. To describe $W$, it is sufficient to specify $r$ and $\K^Y(r)$, so the complexity of $W$ given $Z$ is bounded by $\K^Z(r)+O(\log\K(r))$. The last step: if an effectively open set $W$ of measure $2^{-p}$ has description of complexity at most $q$, all its elements have deficiency at least $p-O(\log p)-q$. (We apply this observation with $p=\K(r)-\K^Y(r)$ and $q=K^Z(r)+O(\log\K(r))$, using $Z$ as an oracle.)

Let us prove two statements used in this argument.

(1)~\emph{Let $x$ be a string. The probability that \textup(uniformly\textup) random oracle $Y$ decreases the prefix complexity of $x$ at least by some $s$, does not exceed $O(2^{-s})$.}

Assume that $\K(x)=t$. Let us first (uniformly) generate a random oracle $U$ and then generate a string according to a priori distribution $\mathbf{m}(x|U)$ using this oracle. Then we get a lower semicomputable discrete semimeasure on strings, and it is bounded by (oracle-free) a priori probability $\mathbf{m}(x)$. The probability to get $x$ in such a process is at least $\Omeg(p2^{-(t-s)})$, where $p$ is the probability to get an oracle $U$ that decreases complexity of $x$ from $t$ to $t-s$ (or more), since the probability to get $x$ using such an oracle is $\Omeg(\mathbf{m}(x|U))=\Omeg(2^{-(t-s)})$. Since $\mathbf{m}$ is maximal, we get $\Omeg(p2^{-(t-s)})= O(\mathbf{m}(x))=O(2^{-t})$, so $p=O(2^{-s})$. [Recall that discrete a priori probabilities $\mathbf{m}(x)$ and $\mathbf{m}(x|U)$ are equal to $2^{-\K(x)}$ and $2^{-\K(x|U)}$ respectively up to a constant factor.]

(2)~\emph{Let $W$ be an effectively open set of measure $2^{-p}$ whose description has prefix complexity at most $q$. Then all elements of $W$ have \textup(expectation-bounded\textup) randomness deficiency at least $p-O(\log p)-q$.}

To construct an expectation-bounded test, let us generate a program $v$ for effectively open set $V$ with probability $\mathbf{m}(v)$, and independently an integer $k$ with probability $\mathbf{m}(k)$. Then let us consider the indicator function $I_V$ that is equal to $1$ inside $V$ and to $0$ outside $V$, and multiply it by $2^k$. We trim the resulting function in such a way that its integral (w.r.t. uniform measure in the Cantor space) is bounded and the function remains unchanged if the integral does not exceed $1$. Then we add all these functions (with weights $\mathbf{m}(v)\cdot \mathbf{m}(k)$); the result is a test (has finite integral). On the other hand, one of the terms corresponds to $V=W$ and $k=p$, and this term remains untrimmed due to our assumptions, so the test is at least $2^p\mathbf{m}(v)\mathbf{m}(p)$, and that is what we need, since $\mathbf{m}(v)\ge \Omeg(2^{-q})$ and $\mathbf{m}(p)=2^{-O(\log p)}$.
 \qed
\end{proof}

\begin{cor}
Let $(r_n)$ be a sequence of strings such that $\C(r_n) \geq n$, and let $\mathcal{C}$ be a non-empty $\mathrm{\Pi}^0_1$-class. Then there exists $Z \in \mathcal{C}$ such that $\C^Z(r_n) \geq n - O(\log n)$. 
\end{cor}

\begin{proof}
Let $(r_n)$ be such a sequence of strings. By the Kucera-G\'acs theorem, there exists some Martin-L\"of random real  $Y$ that computes this sequence. By the basis for randomness theorem, there exist some $Z\in \mathcal{C}$ be such that $Y$ is random relative to $Z$. It remains to apply Lemma~\ref{lem:higuchi}.
\end{proof}

The proof of Theorem~\ref{thm:half-complexity-not-enough} now goes as follows. Let $\phi$ be a sentence not provable in $\PA$. Consider the $\mathrm{\Pi}^0_1$-class of (codes of) complete consistent extensions of $\PA \cup \{ \neg \phi\}$. It is a non-empty class since $\PA \cup \{ \neg \phi\}$ is a consistent computable set of axioms. By the above corollary, let $Z$ be a member of this class such that $\K^Z(r_n) \geq n-O(\log n)$. Let $T$ be the theory coded by $Z$. Since $T$ is complete, it declares the value of Kolmogorov complexity for every string. Let $\C_T$ be this version of Kolmogorov complexity. It is clear that $\C^Z \leq^+ \C_T$ since $\C_T$ is computable with oracle $Z$ and satisfies the quantitative restrictions (no more than $O(2^k)$ strings $u$ have $\C_T(u)<k$). Thus $\C_T(r_n) \geq n-O(\log n)$, so for some $c$ the theory $T$ contains the sentences $``\C(r_n) \geq n-c\log n"$ for all $n$, and also contains $\neg \phi$. Hence the axioms $``\C(r_n) \geq n -c\log n"$ do not prove $\phi$. \qed

\subsection{Independence of random axioms}\label{subsec:independence}

In section~\ref{sec:pp} we added (to \PA) random axioms of the form $\C(x)\ge n-c$ for several randomly chosen strings $x_1,\ldots,x_m$; we noted that if $m2^{-c}\ll 1$, all the added axioms are true with probability close to~$1$. A natural question arises: will these axioms be independent?

Evidently, with positive probability they can be dependent. Imagine that we add axioms $\C(x_1)\ge n-c$ and $\C(x_2)\ge n-c$ for two random strings $x_1$ and $x_2$ obtained by $2n$ coin tosses. It may happen (with positive probability) that $x_1=x_2$ or $x_1$ is so close to $x_2$ that they provably have the same complexity. (For example, the decompressor used in the definition of $\C$ could give the same complexity to strings that differ only in the last bit.) Or it may happen that the first axiom is false, then it implies everything (its negation is provable).

However, the probability of dependence between these two axioms is small. For example, consider the probability $\eps$ that 
$$\PA\vdash (\C(x_1)\ge n -c)  \Rightarrow (\C(x_2)\ge n-c)\eqno(*)$$
for a randomly chosen pair $(x_1,x_2)$. We want to show that this probability is small. Indeed, we can fix $x_2$ in such a way that the probability of $(*)$ for this $x_2$ and random $x_1$ is at least $\eps$. And Theorem~\ref{conservation} says that this is possible only if $\C(x_2)\ge n-c$ is provable (which implies $n=O(1)$ due to Chaitin's theorem) or $\eps \le 2^{-c}$. So for large enough $n$ the probability of $(*)$ does not exceed $2^{-c}$.

Similar results are true for other types of dependence. For example, we may consider three random strings $x_1,x_2,x_3$ and the event 
$$\PA\vdash (\C(x_1)\ge n -c)  \Rightarrow (\C(x_2)\ge n-c)\lor (\C(x_3)\ge n-c)$$
This event also has probability at most $2^{-c}$ for large enough $n$. Indeed, the right hand side implies that $\C(x_2,x_3)\ge n-c-O(1)$ (the pair has large complexity if one of its components has large complexity), and we can use the same argument.

One more type of dependence: consider the event
$$\PA\vdash (\C(x_1)\ge n -c)  \land (\C(x_2)\ge n-c)\Rightarrow (\C(x_3)\ge n-c);\eqno(**)$$
let its probability (for independent random $x_1,x_2,x_3$) be $\varepsilon$. Then for some $x_3$ the probability of this event (for random $x_1$ and $x_2$) is at least $\eps$. Then we can use the same argument as in the proof of Theorem~\ref{conservation}. We can prove in \PA\ that the fraction of pairs $(x_1,x_2)$ such that left hand side of the implication is false, does not exceed $2\cdot 2^{-c}$ (each of the two conjuncts is false with probability at most $2^{-c}$). Therefore, if the fraction of pairs $(x_1,x_2)$ such that $(**)$ happens is greater that $2\cdot 2^{-c}$, we can form a disjunction and then prove $\C(x_3)\ge n-c$ without additional axioms.

Similar reasoning can be applied to other kind of dependencies, so three random axioms
$$\C(x_1)\ge n-c, \quad \C(x_2)\ge n-c,\quad \C(x_3)\ge n-c$$
are independent with probability $1-O(2^{-c})$ for sufficiently large $n$. Here by the independence of the statements $T_1,T_2,T_3$ we mean that each of $8$ possible combinations of the form 
$$(\lnot) T_1,\quad (\lnot) T_2, \quad (\lnot) T_3$$
(with or without negations) is consistent with \PA.

A similar result (with the same proof) is true for any constant number of randomly chosen axioms: 

\begin{thm}\label{thm:independence}
Fix some constant $m$. Let $x_1,\ldots,x_m$ be $m$ independent uniformly randomly chosen strings of length $n$, and consider $m$ statements
    $$
\C(x_1)\ge n-c, \quad \C(x_2)\ge n-c,\ \ldots\ , \C(x_m)\ge n-c.
    $$
For large enough $n$ they are \PA-independent \textup(all $2^m$ combinations of these statements, with negations or not, are consistent with \PA\textup) with probability $1-O(2^{-c})$, where the constant in $O$-notation depends on $m$ but not on~$n$ and~$c$.

\end{thm}

\subsection{The strange case of disjunction}

So far, the results we have established about the axiomatic power of Kolmogorov complexity were related to its computability-theoretic properties. In this section, we present an interesting example of a setting where the axiomatic power of a family of axioms is (in some sense) weaker than its computational power. This family consists of axioms of type ``$\C(x)=n_1 \vee \C(x)=n_2$" where one axiom of this type is given for each~$x$. The following result of Beigel et al.\ tells us that having access, for each~$x$, to a pair of possible values of~$\C$ is enough to reconstruct~$\C$:

\begin{prop}[\cite{BeigelBFFGLMST2006}]
Let $f:\fs \rightarrow \N^2$ be a function such that for each~$x$, if $f(x)=(n_1,n_2)$ then $\C(x) \in \{n_1,n_2\}$. Then the function~$f$ Turing-computes the function~$\C$. 
\end{prop}

Based on that this results and the results presented so far in the paper, one could conjecture that if for each~$x$ we are given a (true) axiom of type ``\,$\C(x)=n_1 \vee \C(x)=n_2$\,'', then we are able to prove all true statements of type ``\,$\C(x)=n$\,''. Surprisingly, this turns out to be false. 

 \begin{thm}\label{thm:strange-disjunction}
Let $\phi$ be a formula which is not provable in~$\PA$. There exists a family~$F$ of true axioms of type ``\,$\C(x)=n_1 \vee \C(x)=n_2$'', where one such axiom is given for any~$x$, such that $\PA \cup F$ does not prove $\phi$.
\end{thm}

\begin{proof}
Since $\phi$ is not provable in $\PA$, the theory $\PA+\lnot\varphi$ is consistent and has some model $\mathfrak{M}$. In this model a formula that defines Kolmogorov complexity function, determines some function $\mathfrak{C}\colon \mathfrak{M}\to\mathfrak{M}$. Note that for standard natural numbers $n\in\mathfrak{M}$ the values $\mathfrak{C}(n)$ are standard (since $\C(x)\le \log x +O(1)$ is provable in $\PA$). The value $\mathfrak{C}(x)$ may coincide with $\C(x)$ or they may differ (in this case $\mathfrak{C}(x)$ is smaller, since a standard description for $x$ remains valid in all models of \PA). Then we add the axioms 
     $$
 \C(x)=\overline {\C(x)}\ \lor \ \C(x)=\overline{\mathfrak{C}(x)}
     $$ 
(containing numerals both for true value $\C(x)$ and $\mathfrak{M}$-value $\mathfrak{C}(x)$, for all strings $x$) to $\PA$. Both the standard model and $\mathfrak{M}$ are models of this theory. Therefore, these axioms are true in the standard model but do not imply $\phi$.
\qed
\end{proof}


\subsection{Axioms on conditional complexity}

What happens if we switch to conditional complexity and add true statements of the form $\C(x|y)\ge n$? The unconditional complexity is a special case of conditional one, so if we add all true statements of this form, we can prove all true $\mathrm{\Pi}_1$ statements.

However, there is an important difference between conditional and unconditional case. The following theorem shows that now we do not need unbounded values of $n$ to get all true $\mathrm{\Pi}_1$-statements.

\begin{thm}\label{thm:cc}
There exists some constant $c$ such that $\PA$ together with all true statements of the type $\C(x|y)\ge c$ proves all true $\mathrm{\Pi}_1$ statements.
\end{thm}

\begin{proof}
   Strangely, the proof is quite indirect here. It use the results from recursion theory about DNC (=diagonally non-computable) functions saying that (1)~the mass problem of constructing a DNC function is equvalent to the mass problem ``given $n$, construct some string of complexity at least $n$'', and that (2)~every enumerable oracle that computes DNC function, is Turing complete. Since we need to translate these results from computation language to proof language, we need to reproduce them first for our special case.
   
\begin{lemma}
There exist a constant $c$ with the following property: having access to an oracle that for every string $x$ gives us some string $y$ such that $\C(y|x)>c$, we can for every $n$ compute a string of complexity at least~$n$.
\end{lemma}   
  
To prove the lemma, fix some programming language. Given some program $p$ and input $u$, we cannot say whether the computation of $p$ on input $u$ terminates. However, we know that if it terminates, the output will have $O(1)$-complexity conditional to $p$ and $u$. The constant on $O(1)$ depends on the programming language (and on the choice of specific complexity function), but not on~$p$ and~$x$. So, having an oracle that for given $x$ produces $y$ such that $\C(y|x)>c$ where $c$ is a bigger constant, we are able for every $p$ and $u$ specify some $y$ that is guaranteed to be different from the output of $p$ on $u$ (if this output exists).  

Our task is, however, more difficult: we want to construct a string of complexity greater than $n$, so we need this string to be different from the outputs of many computations (for all programs of length $n$ or less). This can be done as follows: we may assume that the outputs of computations are not strings but infinite sequences of strings that have only finitely many non-empty terms. Such ``sequences with finite support'' form a countable set and one can establish a computable one-to-one correspondence between such sequences and strings. Now, having finitely many computations whose outputs are such sequences, we construct a sequence that differs from the first computation in the first term, from the second computation in the second term, etc. This can be done using the oracle, as we have seen above. Lemma is proven.\qed

Returning to Theorem~\ref{thm:cc}, consider the following process that uses an oracle providing exact values for $\C(\cdot|\cdot)$. Given some $n$, we construct the string of complexity greater than $n$ as described above. When a string different from the output of some program $p$ on some input $u$ is needed, we take the \emph{first} string $y$ such that $\C(y|u,p)$ exceeds some constant $c$ (large enough and fixed in advance). Finally we produce some string $r_n$ that has complexity greater than $n$. 

While looking for the first strings (denoted by $y$ in the previous paragraph) we observe that all the previous strings have complexity (with required conditions) less than $c$. Consider the time $t$ needed to establish this fact, i.e., some $t$ such that not only true conditional complexity $\C$ but also its upper bound $\C^t$ (complexity with time bound $t$) becomes less than~$c$. Let us show that every number greater than $t$ has complexity at least $n-O(\log n)$. Indeed, knowing some number $t'>t$ and $n$, we do not need the oracle anymore, since we can use $\C^{t'}$ instead of $\C$ and get the same strings $y$. This procedure is computable, so it cannot increase complexity more than by $O(1)$, and we get a string of complexity greater than $n$. Since we need only $O(\log n)$ bits to specify $n$, the number $t'$ should have complexity at least $n-O(\log n)$.  As before, this implies that $t$ steps are enough for every terminating program (without input) of size $n-O(\log n)$ to terminate, otherwise this program would describe the number of steps needed for termination, and this number $t'$ would be greater than $t$ and still have small complexity.

Now we need to formalize the reasoning above in \PA. Having all true statements of the form $\C(u|v)>c$ as axioms, we provably know the first strings $y$ that are found during the described process. (Indeed, we can also prove that previous strings have small complexity, since this is an existential statement.)  We can also prove that the process intended to generate a string of complexity greater than $n$, achieves its goal. Then we can provably establish the value of $t$, and also prove that every number greater than $t$ has complexity $n-O(\log n)$ (with some specific constant in the $O$-notation). Finally, we prove that every program of size $n-O(\log n)$ that does not terminate in $t$ steps, never terminates, therefore proving all true $\mathrm{\Pi}_1$ statements of complexity at most $n-O(\log n)$. Since $n$ is arbitrary, we can prove all true $\mathrm{\Pi}_1$-statements.\qed
\end{proof}

\begin{rem}
  It would be nice to find a more direct way to construct string of arbitrary high complexity if we know all the pairs $(u,v)$ such that $\C(u|v)>c$. One can try to start with some $x_0$ that has complexity at least $c$, then find  some~$x_1$ such that $\C(x_1|x_0)>c$, then $x_2$ such that $\C(x_2|x_0,x_1)>c$, but this does not work because in the formula for the complexity of pair for plain complexity we have logarithmic error terms that can compensate $c$, and in the prefix version we also have the prefix complexity in the condition.
\end{rem}

\section{Adding information about Martin-L\"of random sequences}\label{section_martin_lof_random}

Up to now we considered additional axioms saying that some \emph{strings} have high complexity. In this section we want to extend \PA\ in a different way and claim for some \emph{infinite sequence} $X$ that $X$ is Martin-L\"of random. Some refinements are needed since an infinite sequence (unlike a string) cannot be made a part of one axiom. There are several ways to do this. Let us consider different possibilities.

\subsection{The theory $\MLR_c(X)$ and its properties}

One natural way to express that $X=x_0x_1\ldots$ is Martin-L\"of random is to make use of the Levin-Schnorr theorem, by fixing some constant~$c$ and to add axioms ``$\K(X \uh n)\geq n-c$'' for all $n$. Here $\K$ stands for prefix complexity (note that in the previous section we considered plain complexity which was more natural then, but all the results of that section also hold with $\K$ in place of $\C$) where $X\uh n$ stands for the $n$-bit prefix $x_0x_1\ldots x_{n-1}$ of $X$. We denote by $\MLR_c(X)$ the theory \PA\ enriched by these additional axioms. 

Let us fix $X$ and consider $\MLR_c(X)$ for different $c$. For small $c$ some of the additional axioms can be false; then their negation is provable and $\MLR_c(X)$ is inconsistent. As $c$ increases, the theory $\MLR_c(X)$ becomes weaker. If $X$ is Martin-L\"of random, then for large enough $c$ all statements in $\MLR_c(X)$ are true and $\MLR_c(X)$ is consistent. Note also that all axioms of $\MLR_c(X)$ are $\mathrm{\Pi}_1$-statements, so in the latter case $\MLR_c(X)$ is a part of the theory considered above ($\PA$ plus all true $\mathrm{\Pi}_1$-statements).

Intuitively, if we believe that ``$X$ is Martin-L\"of random" implies some~$\phi$, then we would expect $\phi$ to be provable in \emph{all} theories $\MLR_c(X)$. The next theorem shows that any such formula~$\phi$ is in fact already provable in~$\PA$. 

\begin{thm}\label{mlr-asympt-c}
Let $X$ be a Martin-L\"of random sequence. If $\phi$ is provable in all theories $\MLR_c(X)$, then $\phi$ is provable in $\PA$. 
\end{thm}

\begin{proof}
Let $X$ be a Martin-L\"of random sequence. We need to show that a formula $\phi$ provable in $\MLR_c(X)$ for all $c$ is actually provable in $\PA$ alone. Assume that $\phi$ is not provable; we will show that every sequence $X$ such that $\phi$ is provable in all $\MLR_c(X)$ is not random. This is done by constructing a Martin-L\"of test that covers $X$. For every $c$ consider the set $A_\phi^c$ of infinite binary sequences: 
$$ A_\phi^c = \{ Y ~ |~  \MLR_c(Y) \vdash \phi  \} $$
This is an effectively open set in the Cantor space (recall that each derivation uses only a finite number of axioms).  We claim that the uniform measure of this set is small. More precisely, $\mu(A^c_\phi) \leq 2^{-c}$, so $A_\phi^c$ forms a Martin-L\"of test that covers $X$ (if $\phi$ is provable in $\MLR_c(X)$ for all $c$). 

To prove that claim, suppose the contrary, i.e., $\mu(A^c_\phi) > 2^{-c}$. Since an effectively open set is a union of intervals, this implies that there exists an integer~$N$ and a set~$S$ of more than  $2^{-c}\cdot 2^{N}$ strings $u$ of length~$N$ such that the formula $\MLR_c(u)\to \phi$ is provable for all $u \in S$, where $\MLR_c(u)$ says that $\K(v)\ge |v|-c$ for every prefix $v$ of $u$.

We can then design a proof strategy in the sense of Section~\ref{sec:prob-proof-strategy}. This strategy starts with capital $2^{-c}$, then proves in $\PA$ that there are at least $(1-2^{-c}) \cdot 2^{N}$ strings $u$ of length~$N$ which make $\MLR_c(u)$ true (this statement is used to prove Levin--Schnorr theorem relating Martin-L\"of randomness and prefix complexity, and is provable in $\PA$). Then the strategy picks a string $u$ of length~$N$ at random and adds the axiom $\MLR_c(u)$. By assumption on the cardinality of $S$, with probability greater than $2^{-c}$ we can prove $\varphi$. By Theorem~\ref{conservation}, this would mean that $\phi$ is already provable in $\PA$. 
    \qed
\end{proof}

However, this theorem does not exclude the possibility that for some $X$ and~$c$ the theory $\MLR_c(X)$ is powerful, even powerful enough to prove all true $\mathrm{\Pi}_1$-statements. The next theorem rules out this possibility. 

\begin{thm}\label{mlr-c}
If the theory $\MLR_c(X)$ is consistent, it does not prove all true $\mathrm{\Pi}_1$-statements. 
\end{thm}
\begin{proof}
We start by the following observation. Suppose that for some random~$X$ and for some $c$ theory $\MLR_c(X)$ proves all $\mathrm{\Pi}_1$-statements while being consistent. Then, using~$X$ as an oracle, one can enumerate all theorems of $\MLR_c(X)$, so with oracle $X$ one can decide which $\mathrm{\Pi}_1$-statements are true and which are false (false $\mathrm{\Pi}_1$-statements can be enumerated without oracle), i.e., $X$ computes the halting problem. So we now can make use of this additional information.
 
Let~$X$ be a random sequence such that $\K(u)\ge |u|-c$ for all prefixes of~$X$, and~$X$ computes $\mathbf{0}'$ (the halting problem). Identifying complete arithmetical theories with infinite sequences (where the value of the $n$-th bit is~$1$ if and only if the $n$-th arithmetical formula, for some canonical order, is in the theory), we see that the set of complete and consistent extensions of $\PA$ is a non-empty $\mathrm{\Pi}^0_1$ class. The Turing degrees of the elements of this $\mathrm{\Pi}^0_1$ class are commonly referred to as $\mathsf{PA}\text{-degrees}$. By the low basis theorem for randomness (see~\cite[Proposition~7.4]{DowneyHMN2005} or~\cite[Theorem 8.7.2]{DowneyH2010}), there exists a complete and consistent extension $T$ of $\PA$ such that $X$ is Martin-L\"of random relative to $T$. Since $T$ is complete and consistent, for each~$n$ there is a unique value~$k_n$ such that $T \vdash ``\K(n)=k_n"$. Let $\HH: \N \rightarrow \N$ be the function $n \mapsto k_n$. Let us call $\aaa$ the Turing degree of $T$ and let us make three observations. First~$\HH$ is computable relative to $\aaa$. Second, we must have $\HH(n) \leq \K(n)$: indeed, if $\K(n) < \HH(n)=k_n$, then the statement ``$\K(n)<k_n$" is true and therefore provable in $\PA$, and a fortiori in~$T$, a contradiction with the definition of $k_n$. Third, since it is provable in $\PA$ that $``\sum_n 2^{-\K(n)} \leq 1"$, this must also hold in $T$ and therefore $\sum_n 2^{-\HH(n)} \leq 1$ (Note that these inequalities can be considered as statements about finite sums.). Since~$\HH$ is computable relative to $\aaa$, Levin's coding theorem indicates that $\K^\aaa \leq \HH+O(1)$. To sum up: $\K^\aaa \leq \HH+O(1) \leq \K+O(1)$. Since $X$ is $\aaa$-random and $\HH \geq \K^\aaa -O(1)$, we know that $\HH(X \uh n)-n \rightarrow +\infty$ (here we use the fact that every $\aaa$-Martin-L\"of random~$X$ has the property $\K^\aaa(X \uh n)-n \rightarrow +\infty$, see~\cite{Chaitin1987} for the unrelativized version) so there is an $N$ such that $\HH(X \uh n) \geq n-c$ for all $n \geq N$. 

Now consider the set of strings $R=\{u : \HH(u) \geq |u|-c\}$.
This set is $\aaa$-computable and contains all initial segments of $X$, except (possibly) the first~$N$. Now, consider $\PA$ with additional axioms $\K(u) \geq |u|-c$ for all prefixes $u$ of $X$ and for all $u\in R$.
This theory (called $T'$ in the sequel) is $\aaa$-computable because~$R$ is $\aaa$-computable and we add only finitely many axioms for prefixes of $X$. Moreover, it \emph{extends} $\MLR_c(X)$ and all $T'$-theorems are true (recall that $\HH \leq K$ by construction). If $\MLR_c(X)$ proved all true $\mathrm{\Pi}_1$-statements, then so would the stronger theory $T'$, hence $T'$ would compute $\mathbf{0}'$. But $T'$ is $\aaa$-computable so this would yield $\aaa \geq_T \mathbf{0}'$. This is impossible: our initial assumption is that $X$ is $\aaa$-random; if $\aaa \geq_T \mathbf{0}'$, then $X$ would be $\mathbf{0}'$-random (a.k.a. $2$-random), but no $\mathbf{0}'$-random sequence can compute $\mathbf{0}'$ (see for example~\cite[Theorem 5.3.16]{Nies2009}), and $X$ does compute $\mathbf{0}'$ by our initial remark. \qed
\end{proof}

\subsection{The theories $\MLR_c'(X)$ and $\MLR_c''(X)$}

Another possibility is to extend the language of \PA\ by a unary functional symbol~$f$. Then we can add one axiom saying that $f(0)f(1)\ldots$ is a binary sequence whose prefixes of every length $n$ have complexity at least ${n-c}$ (note that one single formula is now enough to claim that all strings $f(0)f(1)\ldots f(n)$ have complexity at least $n-c$), and a series of axioms that specify the elements of the sequence: $f(0)=x_0$, $f(1)=x_1$, etc. To make this theory reasonable, we also need to add induction over formulas that contain $f$ (just to prove that $f$ has some prefix of each length). Evidently, this theory, which we denote by $\MLR'_c(X)$, proves all the axioms of $\MLR_c(X)$. 

Moreover, it proves some statements that look stronger than $\MLR_c(X)$. Indeed, for each $n$ we can prove in $\MLR'_c(X)$ the statement $\mathsf{Ext}_c(X \uh n)$ where $\mathsf{Ext}_c(u)$ says that for every $n\ge|u|$ string $u$ is a prefix of some string $v$ of length $n$ such that for every prefix $w$ of $v$ the inequality $\K(w)\ge |w|-c$ holds. (Let $v$ be the prefix of $f(0)f(1)\ldots$ of length $n$.)

So, adding all statements $\mathsf{Ext}_c(X \uh n)$  to \PA, we get some intermediate theory between $\MLR_c(X)$ and $\MLR'_c(X)$ which we denote $\MLR''_c(X)$. Note that this theory, unlike $\MLR'_c(X)$, does not contain additional functional symbol $f$. It is natural to ask how these three theories compare to one another. Below are some answers:

\begin{thm}\label{pseudo_equivalence_between_MLR_notions}
   For every $X$ and $c$, the theory $\MLR'_c(X)$ is a conservative extension of $\MLR''_c(X)$: both theories prove the same arithmetical formulas \textup(without~$f$\textup).  
\end{thm}

\begin{proof}
Assume that $\MLR'_c(X)\vdash \varphi$ and the formula $\varphi$ does not involve the symbol~$f$. We need to prove that $\MLR''_c(X) \vdash \varphi$.  Only finitely many axioms $f(i)=x_i$ are involved in the proof of~$\varphi$ from $\MLR'_c(X)$. Let $N$ be the maximal $i$ that appears in these axioms. Now suppose that $\MLR''_c(X)$ does not prove $\varphi$. Consider a model $\mathfrak{M}$ of $\MLR''_c(X) \cup \{\neg \varphi\}$. We want to interpret $f$ in this model and get a model of $\MLR'_c(X)$ restricted to axioms ``$f(n)=x_n$" for $n \leq N$, in which $\varphi$ is false, thus contradicting the assumption. Note that $\mathfrak{M}$ may be a non-standard model.

Consider the string $x=x_0\ldots x_N$ and all its extensions $y$ with the following property: all prefixes $v$ of $y$ satisfy the inequality $\K(v)\ge |v|-c$. If the axioms of $\MLR_c''(X)$ are true in the standard model, these $y$ form a subtree which is infinite. K\"onig Lemma then guarantees that this subtree has an infinite branch. Moreover, we get a definable branch if we take on each level the leftmost vertex in the subtree  that has arbitrarily long extensions in the subtree. Now we can formalize this argument in \PA\ and observe that the same formula defines some branch (i.e., function $f$) in $\mathfrak{M}$, and all vertices on this branch satisfy the inequality $\K(v)\ge |v|-c$ in $\mathfrak{M}$. This allows us to extend $\mathfrak{M}$ to a model of (restricted) $\MLR_c'(X)$, where $\phi$ is false, thus contradicting our assumption. (Note that induction for formulas containing $f$ is possible in $\mathfrak{M}$ since $f$ is definable.)
   %
\qed
\end{proof}


In contrast to Theorem~\ref{mlr-c}, there \emph{is} a theory $\MLR'_c(X)$ which is consistent \emph{and proves all true $\mathrm{\Pi}_1$-statements}. 
%

\begin{thm}\label{random_number_in_model_2_for_PA}
For every~$c$, there exists a sequence~$X$ such that $\MLR''_c(X)$ is consistent and proves all true $\mathrm{\Pi}_1$-statements. 
\end{thm}

\begin{proof}
In this proof we use again the ideas from the proof of Theorem~\ref{thm:finite-power}. 

Let  
\[
\D_c=\{ X \in \cs \mid \forall n\;  \K(X\upharpoonright n) \geq n-c\}
\]
Note that the set of all finite prefixes of strings in $\D_c$ is a co-c.e.\ set. Denote by~$Z$ be the leftmost path of $\D_c $ ($Z$ can be seen as a Chaitin $\mathrm{\Omeg}$ number.) 

By definition theory $\MLR''_c(Z)$ consists of  axioms $\mathsf{Ext}_c(Z \uh n)$. For each~$n$,  $Z\upharpoonright n$ is the first (in lexicographic order) string~$\sigma$ of length~$n$ such that $\mathsf{Ext}_c(\sigma)$ holds. Since $\mathsf{Ext}_c$ is a $\mathrm{\Pi}_1$ predicate,  we can enumerate the strings~$\tau$ such that $\neg \mathsf{Ext}_c(\tau)$. Denote by $t$ the step in this enumeration when we get the list of all  strings smaller than $Z \uh n$. Given $c$ and  $Z \uh n$ we can compute this number $t$. Moreover, given $\mathsf{Ext}_c(\tau)$ we can \emph{prove} that any time $s \geq t$ has complexity at least~$n-O(\log n)$ (similar to the proof Theorem \ref{thm:finite-power}). The rest of the proof is identical to Theorem \ref{thm:finite-power}. \qed
\end{proof}

However, the theories $\MLR''_c(X)$ are still weak in the sense of Theorem~\ref{mlr-asympt-c}: if a formula $\phi$ is provable in all theories $\MLR''_c(X)$ for a given~$X$, then $\phi$ is already provable in $\PA$. The proof is identical. 

This has the following interesting corollary. 

\begin{cor}
Let $\MLR''(X)$ be the statement: $(\exists c)\; \MLR''_c(X)$ (which can be made in $\PA$). Then $\MLR''(X)$ is conservative over $\PA$.
\end{cor}

\begin{proof}
From the above discussion, if $\phi$ is a formula without constant~$f$ that is not provable in $\PA$, then there exists a constant~$d$ such that $\MLR''_d(X)$ does not prove~$\phi$. Since $(\exists c)\, \MLR''_c(X)$ is provable from $\MLR''_d(X)$, it follows that $(\exists c)\; \MLR''_c(X)$ does not prove $\phi$. \qed
\end{proof}

\subsection{Initial segment complexity of nonrandom sequences}

In Theorem \ref{mlr-c} we proved that if $X$ is a Martin-L\"of sequence such that the axioms ``$\K(X \uh n)\geq n-c$'' are true for all $n$, then the theory consisting of these axioms does not prove all true $\mathrm{\Pi}_1$-statements. What happens if we consider sequences $X$ that are not random, but for which ``$\K(X \uh n)\geq n-c$'' is still true for infinitely many $n$? In constrast to Theorem \ref{mlr-c}, there is a sequence $X$ and a constant $c$ such that all true axioms of the form ``$\K(X \uh n)\geq n-c$'' prove all true $\mathrm{\Pi}_1$-statements.

\begin{thm}
\label{nonrandomsequence}
Fix some constant $c \ge 0$. There exists a sequence $X$ and an infinite set $A \subseteq \N$ such that the theory consisting of axioms
\[ \mbox{``$\K(X \uh n) \ge n - c$''} \]
for all $n \in A$ is consistent and proves all true $\mathrm{\Pi}_1$-statements.
\end{thm}

\begin{proof}
We order all strings by length and then lexicographically. That is, $\sigma < \tau$ if and only if $|\sigma| < |\tau|$, or $|\sigma| = |\tau|$ and $\sigma$ is lexicographically before $\tau$.

We construct $X$ as follows: let $y_0$ be some string with $\K(y_0) < |y_0| - c$. Inductively, let $x_n$ be the first (for the above order) string $x$ that extends $y_n$ with $\K(x) \ge |x| - c$, and let $y_{n+1}$ be some string extending $x_n$ such that 
\[ \K(y_{n+1}) < |y_{n+1}| - (n+1) - c. \]
Note that $x_n$ must exist, as every string has a Martin-L\"{o}f random extension, and for a Martin-L\"{o}f random sequence $Z$
\[ \lim_{n \to \infty} \left( \K(Z \uh n) - n \right) = \infty. \]

Let $\displaystyle X = \lim_{n \to \infty} x_n$. Consider the axioms
\[ \mbox{``$\K(x_n) \ge |x_n| - c$''} \]
for all $n \in \N$. (That is: $A = \{|x_n| : n \in \N \}$ in the statement of the theorem.) We claim that this theory can prove all true $\mathrm{\Pi}_1$-statements. The proof is similar to the proof of Theorem \ref{thm:finite-power}.

Let $A(p)$ be any algorithm. As in Theorem \ref{thm:finite-power}, it is sufficient to prove that for every input $p$ such that $A(p)$ does not terminate, our theory proves this non-termination.

Define $t_n$ to be the first $t$ such that $\K^t(x) < |x| - c$ for all strings $x$ that extend $y_n$ and come before $x_n$ in our order. We prove that from the length $|p|$ of a program $p$, we can compute a number $n$ such that the computation $A(p)$ either terminates in less than $t_n$ steps, or does not terminate at all.

Every string $p$ determines the number of steps needed for the termination of $A(p)$. Knowing $p$ and $y_n$, we find this number $t(p)$ and take the first~$x$ that extends $y_n$ such that $\K^{t(p)}(x) \ge |x| - c$. If $A(p)$ does not halt within $t_n$ steps, then we know that $x = x_n$. On the other hand, for every $p$ such that $A(p)$ terminates we get some string $x$ extending $y_n$ with $\K(x) <  \K(p) + \K(y_n) + O(1)$. By definition, $y_n$ has a low complexity. Consequently
\begin{align*}
\K(x) &< |p| + |y_n| - n - c + O(\log |p|) \\
     &< |x| + |p| - n - c + O(\log |p|) \\
\end{align*}
Given the program size $|p|$, we can find an $n$ that is large enough such that $|p|-n+O(\log |p|)$ is negative. For such an $n$, we know that~$x$ is different from~$x_n$. Hence, whether or not $x$ differs from $x_n$ or not, determines whether $A(p)$ terminates or not. So if $A(p)$ terminates at all, then it must do so in less than $t_n$ steps.

As in the proof of Lemma \ref{finite-power_lemma}, this reasoning can be formalized in PA. Having ``$\K(x_n) \ge |x_n| - c$'' as an axiom, we can prove that $x_n$ is the first string extending $y_n$ such that $\K(x_n) \ge |x_n| - c$. Then, given $y_n$, we can prove the value of $t_n$. Finally, given $p$ and taking $y_n$ for $n$ suitably large, we can prove (doing the above prove inside PA) that $A(p)$ either terminates in $t_n$ steps or does not terminate at all, as required.
\end{proof}

Remark that the sequence $X$ that we constructed in the proof, has arbitrarily large \emph{complexity dips} in between the initial segments $x_n$ with complexity at least $|x_n| - c$. Hence $X$ is not Martin-L\"{o}f random. This is essential by Theorem \ref{mlr-c}. Indeed, even if we choose a Martin-L\"of random $X$ and a constant $c$ small enough such that $\K(X \uh n) > n-c$ is not true for all $n$, it still must be true for all but finitely many $n$. In this case the proof of Theorem \ref{mlr-c} still works to show that the theory consisting of all true axioms ``$\K(X \uh n) > n-c$'' does not prove all true $\mathrm{\Pi}_1$-statements.

For plain complexity, the proof of Theorem \ref{nonrandomsequence} does not work. The reason is that not every string has an extension with high plain complexity.

\begin{question}
\label{questionforplaincomplexity}
Does there exist a sequence $X$ and a theory $T$ consisting of infinitely many axioms of the form 
\[ \mbox{``$\C(X  \uh n) > n - c$''} \]
such that $T$ is consistent and proves all true $\mathrm{\Pi}_1$-statements?
\end{question}

Note that if there does exists such a sequence $X$, then $X$ must be $2$-random. This makes the question quite different from Theorem \ref{nonrandomsequence}, as the sequence constructed in the proof of the theorem was necessarily non-random, whereas Question \ref{questionforplaincomplexity} relates to the properties of random sequences.

Moreover, remark that, although there are no Turing-complete $2$-random sequences, some corresponding theory $T$ might still be Turing complete.

A summary of this section and related results can be found in Figure \ref{summary}.

\begin{figure}[h]
    \centering
    Does there exist $A \subseteq 2^{< \omega}$ such that we can prove all true $\mathrm{\Pi}_1$ sentences with consistent axioms\ldots
    
    \vspace{0.5em}
    
    \renewcommand{\arraystretch}{2}
    \setlength{\tabcolsep}{0.4em}
    \begin{tabular}{|c|c|c|}
         \cline{2-3}
         \multicolumn{1}{c|}{} &
         \parbox{9em}{\begin{center} ``$\C(\sigma) > |\sigma| - c$'' \\ for $\sigma \in A$ \end{center}}  &
         \parbox{9em}{\begin{center} ``$\K(\sigma) > |\sigma| - c$'' \\ for $\sigma \in A$ \end{center}} \\
         
         \hline
         \parbox{11em}{\begin{center} $A$ contains at most one string of each length. \end{center}} &
          \textbf{Yes} &
          \textbf{Yes} \\
         
         \hline
         \parbox{11em}{\begin{center} $A$ contains infinitely many initial segments of a sequence. \end{center}} &
         \parbox{10em}{\begin{center} \textbf{Maybe} \\ \scriptsize Note: axioms imply that sequence is $2$-random \end{center}} &
          \textbf{Yes} \\
         
         \hline
         \parbox{11em}{\begin{center} $A$ contains all initial segments of a sequence. \end{center}} &
         \cellcolor{lightgray} \parbox{10em}{\begin{center} \scriptsize Axioms are never consistent \end{center}} &
         \parbox{10em}{\begin{center} \textbf{No} \\ \scriptsize Note: axioms imply that sequence is $1$-random \end{center}} \\
         
         \hline
    \end{tabular}
    \caption{Summary of results about the strength of theories whose axioms express that certain strings have high complexities.}
    \label{summary}
\end{figure}

\section{Acknowledgements}

The authors would like to thank all the colleagues with whom they discussed these results. Special thanks go to Ilya~Razenshteyn for bringing Richard Lipton's post~\cite{Lipton-blog} to our attention and to Chris Porter for many helpful comments on preliminary versions of this paper. This work was supported by ANR-08-EMER-008 NAFIT and grant EMC ANR-09-BLAN-0164-01.

\bibliographystyle{alpha}

\bibliography{randomness_in_PA}

\newcommand{\etalchar}[1]{$^{#1}$}
\begin{thebibliography}{DHMN05}

\bibitem[BBF{\etalchar{+}}06]{BeigelBFFGLMST2006}
Richard Beigel, Harry Buhrman, Peter~A. Fejer, Lance Fortnow, Piotr Grabowski,
  Luc Longpr{\'e}, Andrej Muchnik, Frank Stephan, and Leen Torenvliet.
\newblock Enumerations of the {K}olmogorov function.
\newblock {\em Journal of Symbolic Logic}, 71(2):501--528, 2006.

\bibitem[BGH{\etalchar{+}}11]{BienvenuGHRS2011}
Laurent Bienvenu, Peter G{\'a}cs, Mathieu Hoyrup, Crist{\'o}bal Rojas, and
  Alexander Shen.
\newblock Algorithmic tests and randomness with respect to a class of measures.
\newblock {\em Proceedings of the Steklov Institute of Mathematics},
  274:41--102, 2011.

\bibitem[CC09a]{CaludeC2009}
Cristian Calude and Elena Calude.
\newblock Evaluating the complexity of mathematical problems, part 1.
\newblock {\em Complex Systems}, 18:267--285, 2009.

\bibitem[CC09b]{CaludeC2010}
Cristian Calude and Elena Calude.
\newblock Evaluating the complexity of mathematical problems, part 2.
\newblock {\em Complex Systems}, 18:387--401, 2009.

\bibitem[Cha71]{Chaitin1971}
Gregory Chaitin.
\newblock Computational complexity and {G}{\"o}del's incompleteness theorem.
\newblock {\em ACM SIGCAT News}, 9:11--12, April 1971.

\bibitem[Cha87]{Chaitin1987}
Gregory Chaitin.
\newblock Incompleteness theorems for random reals.
\newblock {\em Advances in Applied Mathematics}, 8:119--146, 1987.

\bibitem[DH10]{DowneyH2010}
Rodney Downey and Denis Hirschfeldt.
\newblock {\em Algorithmic randomness and complexity}.
\newblock Theory and Applications of Computability. Springer, 2010.

\bibitem[DHMN05]{DowneyHMN2005}
Rodney Downey, Denis Hirschfeldt, Joseph~S. Miller, and Andr{\'e} Nies.
\newblock Relativizing {C}haitin's halting probability.
\newblock {\em Journal of Mathematical Logic}, 5(2):167--192, 2005.

\bibitem[HHSY]{HiguchiHSY}
Kojiro Higuchi, Phil Hudelson, Stephen~G. Simpson, and Keita Yokoyama.
\newblock Propagation of partial randomness.
\newblock Submitted.

\bibitem[JS72]{JockuschS1972}
Carl Jockusch and Robert Soare.
\newblock {$\Pi^0_1$} classes and degrees of theories.
\newblock {\em Transaction of the American Mathematical Society}, 173:33--56,
  1972.

\bibitem[LR11]{Lipton-blog}
Richard~J. Lipton and Kenneth~W. Regan.
\newblock Random axioms and {G}{\"o}del incompleteness.
\newblock Blog post, March 2011.

\bibitem[LV08]{LiV2008}
Ming Li and Paul Vit{\'a}nyi.
\newblock {\em An introduction to {K}olmogorov complexity and its
  applications}.
\newblock Texts in Computer Science. Springer-Verlag, New York, 3rd edition,
  2008.

\bibitem[Nie09]{Nies2009}
Andr{\'e} Nies.
\newblock {\em Computability and randomness}.
\newblock Oxford Logic Guides. Oxford University Press, 2009.

\bibitem[She00]{Shen-uppsala}
Alexander Shen.
\newblock Algorithmic information theory and {K}olmogorov complexity.
\newblock Technical Report 2000-034, Uppsala University, Department of
  Information Technology, 2000.

\bibitem[She06]{euler2006}
Alexander Shen.
\newblock Kolmogorov complexity and proof theory: a question.
\newblock In {\em International Conference ``Methods of Logic in Mathematics.
  III'' (June 1 - 7, 2006). Saint-Petersburg, Russia.}, 2006.

\bibitem[Sip96]{Sipser1996}
Michael Sipser.
\newblock {\em Introduction to the Theory of Computation}.
\newblock PWS Publishing Company, 2nd edition, 1996.

\end{thebibliography}

%

%

%
%
%
%
%

\end{document}